\documentclass[11pt]{article}

\usepackage{amsmath,amssymb,latexsym}
\usepackage{xypic}
\usepackage{verbatim}
\usepackage{amsfonts}
\usepackage{fancyhdr}
\usepackage{indentfirst}
\usepackage{newlfont} 
\usepackage[all]{xy}                            
\usepackage{newlfont}
\usepackage{amsthm}
\theoremstyle{plain}                    
\newtheorem{theo}{Theorem}[section]   
\newtheorem{cor}{Corollary}[section]
\newtheorem{conj}[theo]{Conjecture}
\newtheorem{lem}[theo]{Lemma}
\newtheorem{lemma}[theo]{Lemma}
\theoremstyle{definition}               
\newtheorem{defin}{Definition}[section]
\theoremstyle{remark}                   

\renewcommand{\o}{\mathcal{O}}
\newcommand{\ox}{\mathcal{O}_X}
\renewcommand{\j} {\mathcal{J}}
\newcommand{\la}{\longrightarrow}
\newcommand{\p}{\mathbb{P}^1}
\newcommand{\q}{\mathbb{Q}}

\renewcommand{\k}{K_X}
\renewcommand{\j}{\mathcal{J}}
\newcommand{\round}[1]{\lfloor #1 \rfloor}
\newcommand{\roundup}[1]{\lceil #1 \rceil}
\newcommand{\calO}{{\mathcal O}} 
\newcommand{\Q}{\mathbb{Q}}
\begin{document} 
\title{Effective log Iitaka fibrations for surfaces and threefolds.}

\author{Gueorgui Tomov Todorov\thanks{The author would like to thank Professor Christopher Hacon for suggesting the problem and many useful conversations and suggestions.}}

\maketitle
\begin{abstract} We prove an analogue of Fujino and Mori's ``bounding the denominators'' \cite[Theorem 3.1]{fm} in  the log canonical bundle formula (see also \cite[Theorem 8.1]{shok}) for Kawamata log terminal pairs of 
 relative dimension one.
 As an application we prove that for 
  a klt  pair  $(X,\Delta)$ of Kodaira codimension one and dimension at most three such that the coefficients of $\Delta$ are in a DCC set $\mathcal{A}$,
 there is a natural number $N$ that depends only on $\mathcal{A}$ for which 
$\round{N(K_X+\Delta)}$  induces the Iitaka fibration.  We also  prove a 
birational boundedness result for klt surfaces of general type.
\end{abstract}

\section{Introduction.}

 

Let us start by recalling  Kodaira's  canonical bundle formula for a minimal elliptic surface $f\colon S\to C$ defined over the complex number field: $$ K_S=f^*(K_C+B_C+M_C). $$ The {\em moduli part} $M_C$ is a $\Q$-divisor such that $12M_C$ is integral and $\calO_C(12M_C)\simeq J^*\calO_{{\mathbb P}^1}(1)$, where $J\colon C\to {\mathbb P}^1$
 is the $J$-invariant function. The {\em discriminant} $B_C=\sum_P b_PP$, supported 
by the singular locus of $f$, is computed in terms of the local monodromies around 
the singular fibers
 $S_P$. Kawamata~\cite{kaw97,kaw98}
 proposed an equivalent definition, which does not require classification of the fibers: $1-b_P$ 
is the {\em log canonical threshold} of the log pair $(S,S_P)$ in a neighborhood of the fiber $S_P$. 
 
 A higher dimensional analogue consists of a log klt pair $(X,\Delta)$ and
surjective morphism such
 that the Kodaira dimension of $K_X+\Delta$ restricted to the general fibre is zero. For now let us assume  that $K_X+\Delta=f^*D$ for some $\q$-divisor $D$
 on $Y$. Then we can define the  \emph{discriminant}  or divisorial part on $Y$ for  $K_X+\Delta$ to be the  $\Q$-Weil divisor $B_Y:=\sum_P b_P P$, where $1-b_P$ is the maximal
 real number $t$ such that the log pair $(X,\Delta+tf^*(P))$ has log canonical singularities over the
 generic point of $P$. The sum runs over all codimension one points of $Y$, but it has finite
support. The \emph{moduli part} or \emph{J-part} is the unique $\Q$-Weil divisor $M_Y$ on $Y$ satisfying 
$$ K_X+\Delta= f^*(K_Y+B_Y+M_Y). $$
 
 According to Kawamata \cite[Theorem 2]{kaw}(see also Ambro \cite[Theorem 0.2 (ii)]{ambroSBP} and Fujino \cite{fujino}) we know 
that on some
 birational model $\mu:Y'\la Y$ the moduli divisor $M_{Y'}$  is nef. 

 Some of the main questions
 concerning the moduli part are  the following.

\begin{conj}\label{conj} (\cite[Conjecture 7.12]{shok}) Let $(X,\Delta)$ and  $f:X\la Y$ be as above and let us write as before 
$$
K_X+\Delta=f^*(K_Y+B_Y+M_Y).
$$ 

 Then we have the following
\begin{description}
\item[(1)] (Log Canonical Adjunction) There exists a birational contraction 
$\mu:Y'\la Y$ such that after base change the induced moduli divisor $M_{Y'}$ on $Y'$ is semiample. 
\item[(2)] (Particular Case of Effective Log Abundance Conjecture). Let $X_\eta$ be the generic fibre of $f$. Then $I(K_{X_\eta}+\Delta_{X_\eta})\sim 0$, where 
$I$ depends only on $\dim X_\eta$ and the horizontal multiplicities of $\Delta$.
\item[(3)] (Effective Adjunction)
There exist a positive integer depending only on the dimension of $X$ and the 
horizontal multiplicities of $\Delta$ such that 
$IM_{Y'}$ is base point free on some model $Y'/Y$.
\end{description}
\end{conj}

There is a proof of the above conjecture  by Shokurov and Prokhorov in the case
 in which the relative dimension of $f$ is one (Theorem 8.1 of \cite{shok}). For results towards {\bf (1)} see Ambro \cite{ambro3}. Here we prove that there exist a positive integer $I$ depending only on the dimension of $X$ and the 
horizontal multiplicities of $\Delta$ such that $IM$ is integral when the 
relative dimension   is  one using  ideas of Mori and Fujino \cite{fm} 
(see also \cite{ko}). The main advantage of our proof is that the number $I$  that we produce is explicitly computable.

 Our main interest in Conjecture \ref{conj} is because of its applications towards boundedness results for Iitaka fibrations. When $X$ is of general type 
the existence of a natural number $N$ such that $|NK_X|$ induces the Iitaka fibration is know by  results of C. Hacon and J. M$^\textrm{c}$Kernan(cf. \cite{chris1}) and Takayama(cf. \cite{tak}) following ideas by Tsuji.  Similar results in low dimension when $X$ is not of general type  appear in the recent preprints  \cite{vie,ringler,pacienza}. Here we address the boundedness of Iitaka fibrations in the log case.

\begin{theo}\label{appl} Let $(X,\Delta)$ be a klt log pair of Kodaira codimension one and dimension at most three. Then there is a natural number 
$ N$ depending only on the coefficients of 
$\Delta$ such that  $|\round{N(K_X+\Delta)}|$ induces the Iitaka fibration.
\end{theo}

 The proof of the above Theorem in dimension two  relies on the existence 
of $I$ as in the Conjecture and follows the strategy in Section 6 of \cite{fm}.
 For the proof of the Theorem in dimension three 
we need to bound the  smallest positive number $N$ such that 
$|N(K_X+\Delta)|$ induces  a birational map 
for any log surface of general type with the coefficients of $\Delta$ in a DCC set $\mathcal{A}$   as a function of the DCC set only
 (i.e. $N=N(\mathcal{A})$). 
 This is an interesting question in its own right (cf. \cite{vie}) and we address it in the  last section.  We can show that:
\begin{theo} Let $(X,\Delta)$ be a klt surface
 and assume that  the coefficients of $\Delta$ are in a DCC set $\mathcal{A}$ . Then there is a number $N$ depending only on $\mathcal{A}$ such that $\roundup{m(K_X+\Delta)}$ (and  $\round{m(K_X+\Delta)}$) defines a  birational map for $m\ge N$.
\end{theo}
 The above two Theorems complete the boundedness of Iitaka fibrations of klt pairs of dimension two ( for the case of Kodaira dimension zero see \cite{alex}).
 The proof is based on the fact that by a result  of \cite{alex} (see also \cite{almo}) for these surfaces we have a
 lower bound of the volume which  allows us to produce centres of log canonical singularities of a controlled multiple of 
$K_X+\Delta$. Using standard techniques we reduce to the case where 
the centres are isolated points. In order to achieve this, using ideas of 
M$^\textrm{c}$Kernan \cite{mac} and Tsuji, we 
 produce a morphism to a curve and  we 
 use this morphism
to produce the required sections (cf. \cite{tod}).

\section{Preliminaries.}
\subsection{Notations and Conventions.}

We will work over the field of complex numbers $\mathbb{C}$. 
A $\q$-Cartier divisor $D$ is nef if $D\cdot C \ge0$ for any curve $C$ on $X$.
 We call two $\q$-divisors $D_1, D_2$ $\q$-linearly equivalent $D_1\sim_\q D_2$ if 
there exists an integer $m>0$ such that $mD_i$ are integral and linearly equivalent.
We call two $\q$-Cartier divisors $D_1, D_2$ numerically equivalent $D_1\equiv D_2$ if $(D_1-D_2)\cdot C=0$  for any curve $C$ on $X$.
A log pair $(X,\Delta)$ is a normal variety $X$ and an effective $\q$-Weil 
divisor $\Delta$ such that $\k+\Delta$ is $\q$-Cartier. A projective morphism 
$\mu:Y \la X$ is a log resolution of the pair  $(X,\Delta)$ if $Y$ is smooth and 
$\mu^{-1}(\Delta)\cup\{\textrm{exceptional set of } \mu\}$ is a divisor 
with simple normal crossing support. For such $\mu$ we write   
$\mu^*(\k+\Delta)  =K_Y+\Gamma$, and $\Gamma=\Sigma a_i\Gamma_i$ 
where $\Gamma_i$ are distinct integral divisors. A pair is called 
klt (resp. lc) if there is a log resolution $\mu:Y \la X$ such that in 
the above notation we have $a_i <1$ (resp. $a_i\le 1$). The number $1-a_i$ is called 
log discrepancy of $\Gamma_i$ with respect to the pair  $(X,\Delta)$.  
We say that a subvariety  $V \subset X$ is a log canonical centre if it is the 
image of a divisor of log discrepancy at most zero.
 A log canonical place is a valuation corresponding to a divisor of log discrepancy at most zero.
 A log canonical centre is pure if $\k+\Delta$ is log canonical at the generic point of $V$. If moreover there is a unique log canonical place lying over the generic point of V, then we say that $V$ is exceptional. 
LCS$(X,\Delta,x)$ is the union of all log canonical centres of $(X,\Delta)$ through the point $x$.
We will denote by LLC$(X,\Delta,x)$ the set of all log canonical centres 
containing a point $x \in X$. 

\subsection{Generalities on cyclic covers.}
\begin{defin}\label{cyc} 
\end{defin}
Let $X$ be a smooth variety and $L$ a line bundle on $X$ and $D$ an integral
divisor. Assume that $L^m\sim \ox(D)$. Let $s$ be any rational section and
$1_D$ the constant section of $\ox(D)$. Then $1_D/s^m$ is a rational function 
which gives a well defined element of the quotient group $k(X)^*/(k(X)^*)^m$, 
thus a well defined degree $m$ field extension $k(X)(^m\sqrt{1_D/s^m})$. Let 
$\pi:X'\la X$ denote the normalization of $X$ in the field $k(X)(^m\sqrt{1_D/s^m})$. Then 
\begin{description}
\item[(1)] $\pi_*\o_{X'}=\sum_{i=0}^{m-1}L^{-i}(\round{iD/m})$, and 
\item[(2)] $\pi_*\omega_{X'}=\sum_{i=0}^{m-1}\omega_X\otimes L^i(-\round{iD/m})$.
\end{description}
 
 In particular, if $E$ is any integral divisor then the normalized cyclic cover obtained from  $L^m\sim \ox(D)$ is the same as the normalized cyclic cover obtained from  $(L(E))^m\sim \ox(D+mE)$. If $D$ has simple normal crossing support then $X'$ has only rational singularities.

\subsection{DCC sets}
\begin{defin} A subset $\mathcal{A}$ of $\mathbb{R}$ is said to satisfy the descending chain condition if any strictly decreasing subsequence of elements of 
$\mathcal{A}$ is finite. In this case we  also say that $\mathcal{A}$ is a DCC set.
\end{defin}
 
For the general properties of DCC sets we refer to Section 2 of \cite{almo}.

\begin{defin} A sum of $n$ sets $\mathcal{A}_1,\mathcal{A}_2,...,\mathcal{A}_n$ is defined as
$$
\sum_{i=1}^n\mathcal{A}_i=\{a_1+a_2+...+a_n|a_i\in\mathcal{A}_i\}.
$$
Define also 
$$
\mathcal{A}_\infty=\{0\}\cup\bigcup_{n=1}^{\infty}\sum_{i=1}^n\mathcal{A}.
$$
\end{defin}
 
If $\mathcal{A}$ is a DCC set and it contains only non-negative numbers 
then it is easy to see that $\mathcal{A}_\infty$ is also a DCC set.

\begin{defin} For $\mathcal{A}\subset [0,1]$ we define the derivative set 
$$
\mathcal{A}'=\{\frac{n-1+a_\infty}{n}|n\in\mathbb{N},a_\infty\in\mathcal{A}_\infty\cap[0,1]\}\cup\{1\}.
$$
\end{defin}

 It is easy to verify that if  $\mathcal{A}$ is a DCC set then so is $\mathcal{A}'$.

\section{Bounding the moduli part.}

We start by describing the moduli part as it appears in \cite{ko}.
Let $f:(X,R)\la Y$ be a proper morphism of normal 
varieties with generic fibre $F$ and $R$  a 
 $\q$-divisor such that 
$K_X+R$ is $\q$-Cartier and assume that $(F,R_{|F})$ is lc and that $K_F+R_{|F}\sim_{\q}0$.
 Let $Y^0\subset Y$ and $X^0=f^{-1}(Y^0)$ be open subsets such that $K_{X^0}+R^0\sim_{\q} 0$ where $R^0:=R_{|X^0}$ (cf. \cite[Lemma 8.3.4]{ko}).
Write $R^0=D^0+\Delta^0$ with $D$ integral and $\round{\Delta}=0$.

 Assume that $X^0, Y^0$ are smooth and $R^0$ is relative simple normal crossing over $Y^0$.

 Define $V=\o_{X^0}(-K_{X^0}-D^0)$. Let $m$ be (the smallest) positive integer such that $m\Delta^0$ is an integral divisor. Then we have an isomorphism 
$$ 
V^{\otimes m}\cong \o_{X^0}(m\Delta^0),
$$
which defines a local system $\mathbb{V}$ on $X^0\setminus R^0$ (cf. \cite[Definition 8.4.6]{ko}).

 Assume also that $Y$ is smooth, $Y\setminus Y^0$ is a simple normal crossing divisor and that $R^{\dim F}f_{*}\mathbb{V}$ has only unipotent monodromies.
Then the bottom piece of the Hodge filtration of $R^{\dim F}f_{*}\mathbb{V}$ has a natural extension to a line bundle $J$. Set $J(X/Y,R)$ to be  the divisor class corresponding to $J$.

 If the smoothness, normal crossing, and unipotency  assumptions above are not satisfied, take a generically finite morphism $\pi:Y'\la Y$ and a resolution of the main component $f':X'\la X\times_{Y} Y'\la Y'$ 
for which  the assumptions hold and $R'$ the corresponding divisor.
 Then define 
$$J(X/Y,R)=\frac{1}{\deg \pi}\pi_*J(X'/Y',R').$$

We need the following definition.

\begin{defin}(\cite[Definition 8.4.2]{ko}) Assume that $(X,R)$ is lc and $K_X+R\sim_\q 0$ and write $R=R_{\ge 0}-R_{\le 0}$ as the difference of its positive and negative parts. Define 
$$
p_g^+(X,R):=h^0(X,\o_X(\lceil R_{\le 0} \rceil)).
$$
\end{defin}
\begin{theo}\label{kodform}(\cite[Theorem 8.5.1]{ko}) Let $X, Y$ be normal projective varieties and let $f:X\la Y$ a dominant morphism with generic fibre $F$. Let $R$ 
be a $\q$-divisor on $X$ such that $K_X+R$ is $\q$-Cartier and $B$ a reduced divisor on $Y$. Assume that
\begin{description}
\item[(1)] $K_X+R\sim f^*($some $\q$-Cartier divisor on $Y),$
\item[(2)]$p_g^+(F,R_{|F})=1$, and 
\item[(3)] $f$ has slc fibres in codimension 1 over $Y\setminus B$ (cf. \cite{ko}).
\end{description}
Then one can write 
$$
K_X+R\sim_\q f^*(K_Y+J(X/Y,R)+B_R),\textrm{   where}
$$
\begin{description}
\item[(i)]  $J(X/Y,R)$ is the moduli part defined above,
\item[(ii)] $B_R$ is the unique $\q$-divisor supported on $B$ for which there is a codimension $\ge 2$ closed subset $Z \subset Y$ such that 
$(X\setminus f^{-1}(Z),R+f^*(B-B_R))$ is lc and every irreducible component of $B$ is dominated by a log canonical centre of $(X,R+f^*(B-B_R))$.
\end{description}
\end{theo}

 Let  $(X_1,R_1)$, $f_1:X_1\la Y_1$ and $B_1$  be a pair satisfying the assumptions of 
Theorem \ref{kodform} and $R_1$ effective on the general fibre.
 Assume 
furthermore that the relative dimension of $f_1$ is one and that $(X_1,R_1)$ is klt.  Then the  
following holds.

\begin{theo}\label{main} There exist an integer $N$ depending only on the horizontal 
multiplicities of $R_1$ such that the divisor $NJ(X_1/Y_1,R_1)$ is integral.
\end{theo}

 We are going to prove the theorem in the case when the restriction of $\Delta$ to the general fibre of $f$ is non-trivial. When the restriction is trivial the Theorem follows as in \cite[Theorem 3.1]{fm}.

{\bf Step I.} We start with some harmless reductions. Cutting by hyperplanes we can reduce to the case when $Y_1$ 
is a curve. From Step 1 of the proof of \cite[Theorem 8.5.1]{ko},
  we can reduce to the case with normal crossing assumptions, that is we can assume that  $X_1, Y_1$ are 
 smooth, $R_1+f_1^*B_1$ and $B_1$ are snc divisors, $f_1$ is  smooth over $Y_1\setminus  B_1$
 and $R_1$ is  relative snc divisor over   $Y_1\setminus  B_1$.

 By Step 2 of the same proof  we can assume also that $B_1=B_{R_1}$.

{\bf Step II - Galois cover of $Y_1$.} By \cite[(4.6) and (4.7)]{mo}  there is a finite Galois cover $\pi:Y\la Y_1$ with Galois group $G$, such that for the
 induced morphism $f:X \la Y$ every possible local system $R^{\dim F}f_{*}\mathbb{V}_j$ 
 (the difference being given by a choice of isomorphism between two line bundles, compare  \cite[Remark 8.4.7]{ko}) has unipotent monodromies around 
every irreducible component of $B$. Note that we can also arrange that $G$ acts on $X$.

 Thus by  \cite[Theorem 8.5.1]{ko} we have that 
the moduli part $L:=J(X/Y,R)$ is integral. Here $\pi_{X}^*(K_{X_1}+R_1)=K_X+R$ and $R=(\pi_{X})_*R_1$.

{\bf Step III - Constructing the right cyclic cover.} 
There is a unique way to write 
$$
R-f^*B=\Delta-G+E,
$$
where $\Delta$ is effective, $\round{\Delta}=0$, the divisors $E$ and $G$ are integral and  vertical.

 Let $M=f^*(K_Y+L+B)-(K_X+E-G+f^*B)$. Notice that $M$ is integral and that $\Delta\sim_\q M$. Pick $m>0$ such that $m\Delta$ is integral on the general fibre $F$. Such $m$ depends only on the horizontal multiplicities of $\Delta$. Since $m\Delta_{|F}\sim m M_{|F}$  there is an integral divisor $D$ on $X$ such that  $D\sim mM$ and $D_{|F}=m\Delta_{|F}$.


 Construct the cyclic covering $h:Z\la X$ corresponding to $D \sim mM$ and let $X^0$ be as in Definition \ref{cyc}. After possibly changing the birational model of   $X$ we can assume that $D$ is simple normal crossing and  $Z$ has rational singularities. We have the following diagram.
$$
\xymatrix{
  &               &Z\ar[dl]_h \ar@/^/[ddl]^{h'}\\
X_1\ar[d]^{f_1}&X\ar[l]_{\pi_X} \ar[d]^f&\\
Y_1           &\ar[l]_{\pi} Y&
}
$$
The restriction of $\pi_X$ to 
$X^0$ gives one of the cyclic covers used in the construction of the local systems 
$R^{\dim F}f_*\mathbb{V}_j$.

 We have that 
$$ 
h_*\omega_{Z}=\sum_{i=0}^{m-1}\o_X(K_X+iM-\round{iD/m}).
$$

{\bf Step IV - The G-action on $h'_*(\omega_{Z/Y})$.}

  We now proceed as \cite[3.8]{fm}. By the pull-back property \cite[Proposition 8.4.9 (3)]{ko} 
 we have that $L=\pi^*J(X_1/Y_1,R_1)$. 
Let $P\in Y$ and   localize 
everything 
in a neighborhood of $P_1=\pi(P)$ and $P$, and let $e$ be 
the ramification index at $P$. Let $z_1$ be a local coordinate for the germ 
$(Y_1,P_1)$ and $z=(z_1)^{1/e}$ for $(Y,P)$. 
Since  the divisor $D$ is $\mu_e$-equivariant over an open set $Y_0\subset Y$  there is a  group $G_0$  acting  on $Z_{|Y_0}$  
which fits in the sequence  $0\rightarrow\mu_m\rightarrow G_0\rightarrow\mu_e\rightarrow 0$. 
In fact if locally $X$ is Spec$A$ then $Z$ is Spec$A[\phi^{1/m}]$ where $\phi$ is a local equation of $D$. 
Since locally $D$ is $\mu_e$-equivariant $\mu_e$ acts on $\phi$ by multiplication by $e$-th root of unity $\epsilon$ and $\mu_m$ acts on $\phi^{1/m}$ by a multiplication by an $m$-th root of unity $\varepsilon$ there is $\mu_m\rtimes\mu_e$ action on $Z$. Thus we can define a $\mu_{er}$-action on 
the local systems $R^1h'_*\mathbb{V}_j$ where $r=m/(m,e)$ and hence on the canonical extension 
$h'_*\omega_{Z/Y}\otimes\mathbb{C}(P)$. The action on the summand  $
L\otimes\mathbb{C}(P)\subset h'_*\omega_{Z/Y}\otimes\mathbb{C}(P)$ is by a character $\chi_{P}$.By Lemma \ref{trivial},  $NJ(X_1/Y_1,R_1)$ is a divisor 
if and only if the 
character $\chi_P^N$ is trivial for every $P\in Y$. 

  Let $E$ be the general fibre of $h'$. Then by \ref{cyc} we have that $$h^0(E,\omega_E)=h^0(F,\sum_{i=0}^{m-1}\omega_F^{1-i}(-\round{i\Delta_{|F}}))\le (m-1)^2.$$
 
 Reasoning as in \cite[3.8]{fm} if $l$ is the order of $\chi_P$, then $\varphi(l)\le (m-1)^2$, where $\varphi(l)$ is the Euler function. 
Set $N(x)=\textrm{lcm}\{l|\varphi(l)\le x\}$. Then for $N_1=N((m-1)^2)$, the divisor $N_1J(X_1/Y_1,R_1)$ is integral.
\qed

{\em Remark.} Note that the number above is easy to compute explicitly. This is 
then main advantage of our approach.

\subsection{Auxiliary Lemma.}

 Let $Y$ be a smooth curve and let $h:Y'\la Y$ be a finite Galois cover with group $G$. 
Let $D$ be a $\q$-divisor on $Y$ such that $D'=h^*D$ is Cartier. 
For $p'\in Y'$ let $G_{p'}$ be the stabilizer. 
We have that $G_{p'}$ acts on $\o_{Y'}(D')\otimes \o_{P'}$ 
via a character $\chi_{p'}:G_{p'}\la \mathbb{C}$. In this setting 
we have the following lemma due to Fujino and Mori \cite{fm}.
\begin{lem}\label{trivial}(cf. \cite{fm}) For an integer $N$ the divisor $ND$ 
is integral if and only if for each $p'\in Y'$ the character $\chi_{p'}^{N}$ is trivial.
\end{lem}

\section{Iitaka fibrations for surfaces of log Kodaira dimension one.}
  
 In this section we prove Theorem \ref{appl} in dimension two. We start with the following lemma.

\begin{lemma}\label{simDCC} Let $(X,\Delta)$ be a klt pair of dimension $n$ where the coefficients 
of $\Delta$ are in a DCC set $\mathcal{A}\subset [0,1]$. Let $f:X\la Y$
 be a surjective projective morphism such that for the general 
fibre $F\cong\p$  we have that  $(K_X+\Delta)_{|F}\sim_\q 0$. 
Then the set $\mathcal{B}$ of coefficients of the horizontal components of $\Delta$ is finite. In particular there is an integer $m=m(\mathcal{B})$  that clears all the denominators of the horizontal components.
\end{lemma}

\begin{proof} We can describe  $\mathcal{B}$ as the set 
$\{b\in\mathcal{A}|b+a=2,\textrm{for some }a\in\mathcal{A}_\infty\}$.
 $\mathcal{B}$ is a subset  of a bounded  DCC set,
 so it is itself a bounded DCC set. 
If $\mathcal{B}$ is infinite, then there is an increasing infinite sequence.
 But this would give a decreasing infinite sequence in $\mathcal{A}_\infty$,
 which is   impossible since $\mathcal{A}_\infty$ is a DCC set.
\end{proof}
 
\begin{theo}\label{surf} Let $(X,\Delta)$ be a klt pair of dimension
 two and assume that the coefficients of $\Delta$ are in a 
DCC set of rational numbers $\mathcal{A}\subset [0,1]$. Assume that 
$\kappa(K_X+\Delta)=1$. Then there is an explicitly computable
 constant $N$ depending only on 
the set $\mathcal{A}$  such that
$\round{ N(K_X+\Delta) }$ induces the Iitaka fibration.
\end{theo} 
\begin{proof}
To prove the theorem we are free to  change the birational
 model of $(X,\Delta)$ (without changing  the coefficients of $\Delta$).
 So after running the Log Minimal Model Program we can assume that $K_X+\Delta$
 is nef. Log abundance for surfaces implies that $K_X+\Delta$ is semiample. 
Therefore there
 exists a positive integer $k$ such that $|\round{k(K_X+\Delta)}|
$ 
defines the Iitaka fibration $f:X\la Y$. 
 The morphism $f:X\la Y$ for $K_X+\Delta$ satisfies the
 hypothesis of Theorem \ref{kodform}, and hence we can write 
$$
K_X+\Delta\sim_\q f^*(K_Y+B+J).
$$
  By 
replacing the morphism $f:X \la Y$ by an appropriate model  we can assume that we have an isomorphism 
$$
H^0(X,\round{n(K_X+\Delta)})\cong H^0(Y,\round{n(K_Y+B+J)})
$$
for every natural number $n$ divisible by $m$ as of Lemma \ref{simDCC} and $\Delta$ is simple normal crossing over the generic point of $Y$ ( cf. \cite[Theorem 4.5]{fm}).
Here $Y$ is a smooth curve. The coefficients of $B$ are in a DCC set
 depending only on $\mathcal{A}$  (cf. \cite[ Remark 3.1.4]{ambroAC}).
  We follow the argument in Section 6 of \cite{fm}   
to compute an integer $N$ depending only on $\mathcal{A}$ for which $\round{N(K_Y+B+J)}$
 is an ample divisor.
 By Theorem \ref{main} there is an integer $m$, depending only on the 
DCC set $\mathcal{A}$ by Lemma \ref{simDCC}, 
for which 
$mJ$  is  integral. Also note that $\round{B}\ge 0$.

We  treat three cases.
\begin{description}
\item[Case 1] ($g\ge2$). For $N=3m$ we obtain that  deg$\round{N(K_Y+B+J)}\ge 2g+1$ and so 
the divisor in question is ample.
\item[Case 2] ($g=1$). We have that $\deg(J+B)>0$ and the coefficients  of $m(J+B)$ are of the form integer plus an element in  a fixed DCC set. Hence there is a positive constant $c=c(\mathcal{A})$ such that the multiplicity at of 
$m(J+B)$ at some point is greater than $c$. Then for $N>\frac{3}{c}$ we have that $\deg\round{N(J+B)}\ge3$. 
\item[Case 3] ($g=0$). In this case we have to find an integer $N$ such that deg$\round{N(J+B)}-2N>0$. This follows immediately form Lemma \ref{DCCr}.
\end{description}
\end{proof}       

\begin{lemma}\label{DCCr} For any set of elements $a_i$ in a DCC set $\mathcal{A}\subset (0,1)$ such that $-2+\sum_{i=1}^na_i>0$ there is an integer $N=N(\mathcal{A})$ such that $-2N+\sum_{i=1}^n\round{Na_i}>0$.
\end{lemma}
\begin{proof} We proceed by induction on $n$. Let $c$ be any number 
$0<c<\min\mathcal{A}$ and  $k$ such that $0<k<\min\{\mathcal{A}_\infty\cap(2,\infty)\}-2$.
 The base case is $n=3$ and then it is enough to take $N>\frac{4}{k}$. In fact 
$$
\round{Na_1}+\round{Na_2}+\round{Na_3}\ge
\round{Na_1}+\round{Na_2}+\round{2N}-\round{N(2-a_3)}-1.
$$
But $N(a_1+a_2+a_3-2)>4$ hence $\round{Na_1}+\round{Na_2}-\round{N(2-a_3)}>2$ and so the desired inequality follows.

 For the inductive step suppose that $\sum_{i=1}^na_i\ge 3$ and order the $a_i$ so that $a_i\le a_{i+1}$. Then $\sum_{i=1}^{n-1}a_i>2$ and the assertion follows by induction. If not we have that $\sum_{i=1}^na_i < 3$  and hence $n<\frac{3}{c}$. It suffices to take $N>\frac{3+c}{ck}>\frac{n+1}{k}$ since then 
$$
\sum_{i=1}^n\round{Na_i}-2N\ge \round{\sum_{i=1}^nNa_i-2N}-n+1\ge 2.
$$
\end{proof}

\section{Iitaka fibration for threefolds of log Kodaira dimension two.}

 In this section we complete the proof of Theorem  \ref{appl} by proving it 
 in dimension three.

 \begin{theo}\label{3folds} Let $(X,\Delta)$ be a klt pair of dimension 
three and assume that the coefficients of $\Delta$ are in a DCC set of rational  number  
$\mathcal{A}\subset [0,1]$. Assume that $\kappa (K_X+\Delta)=2$. Then there is a constant $N$ depending only on 
the set $\mathcal{A}$  such that
$\round{ N(K_X+\Delta) }$ induces the Iitaka fibration. 
\end{theo}
\begin{proof} Performing the same type of reductions  in the proof of Theorem \ref{surf} 
  we assume that we are in the case when we have a morphism $f:X\la Y$ where $Y$ is a surface, $\Delta_{|F}$ is non-trivial and  we have an isomorphism  
$
H^0(X,\round{n(K_X+\Delta)})\cong H^0(Y,\round{n(K_Y+B+M)})
$
for every $n$ sufficiently divisible.
Here the divisor $K_Y+B+M$ is big, the coefficients of $B$ are in a DCC set
 depending only on $\mathcal{A}$ (cf. \cite[ Remark 3.1.4]{ambroAC}), $M$ is nef. Now take $n$ also  divisible by  by $l$ where  $l$ is an integer 
such that 
 $lM$ is integral and $|lM|$ is base point free. The  integer $l$ depends 
only on the DCC set $\mathcal{A}$. Such $l$ exists  by the case of   Conjecture \ref{conj}
 that is proven in \cite[Theorem 8.1]{shok}.

Notice that $(Y,B)$ is klt  by  \cite[Theorem 3.1]{ambroSBP} and also by
 \cite[Corollary 7.17]{shok}. 
The divisor $lM$ is base point free so we can replace it with a 
linearly equivalent divisor in $M'$, such that the the pair $(Y,B+\frac{1}{l}M')$ 
is klt and $H^0(Y,\round{n (K_Y+B+\frac{1}{l}M'}) = H^0(Y,\round{n(K_Y+B+M)})$ 
for every natural number $n$ divisible by $l$.

Now define the DCC set $\mathcal{B}=\mathcal{A'}\cup\{\frac{1}{l}\}$. Observe 
that $\mathcal{B}$ depends only on $\mathcal{A}$. Define $B_1=B+\frac{1}{l}M'=\sum_i b_iB_i$ where $B_i$ are distinct irreducible divisors.
 By \cite[Theorem 4.6]{almo} there is a computable constant $\beta$ that 
depends only on $\mathcal{B}$ such that $K_Y+(1-\beta)B_1$ is 
a big divisor. Let $b$ be the minimum of the set $\mathcal{B}$ and let 
$k=\roundup{\frac{1}{b\beta}}$. Then  define 
$B'=\sum_ib_i'B_i$ where $b_i'=\frac{\round{kb_i}}{k}$. We have that  the 
divisor $K_Y+B'$ is big with coefficients in the DCC set 
$\mathcal{C}=\{\frac{i}{k}|i=1,\ldots, k-1\}$. Also we have the inclusion 
$H^0(Y,\round{m(K_Y+B')}) \subset H^0(Y, \round{m(K_Y+B+\frac{1}{l}M')})$ 
for every $m$.

Now Theorem \ref{birsurf}  implies that there is a number $N'$ 
depending only on $\mathcal{A}$ such that $\roundup{m(K_Y+B')}$ 
defines a  birational map for $m\ge N'$. Define $N=kN'$. Then we have that $H^0(Y,\roundup{N(K_Y+B')})=H^0(Y,\round{N(K_Y+B')})\subset H^0(Y,\round{N(K_Y+B+\frac{1}{l}M')})$ and hence the
theorem follows.

\end{proof}

\section{Birational boundedness for log surfaces of general type.}
In this section we prove that for a surface pair $(X,\Delta)$ 
of log general type with the coefficients of $\Delta$ in a DCC set $\mathcal{A}$
 there is a number $N$ depending only on $\mathcal{A}$ such that the linear system 
 $|\roundup{N(K_X+\Delta)}|$
gives a birational map. Again by  \cite{alex} or  \cite[Theorem 4.8]{almo} we have that vol$(K_X+\Delta)>\alpha^2$ for some $\alpha$ depending 
only on the DCC set $\mathcal{A}$.
 We are going to use this lower 
bound of the volume to create a log canonical centre.
 The good case is when the  volume of the restriction of 
$K_X+\Delta$ to the log canonical centre is large. 
Then we can proceed by cutting down 
 the log canonical centre to a point and we generate
 a section of an appropriate multiple of $K_X+\Delta$. 
If the volume of the restriction is smaller then 
 we are going to proceed as in \cite{tod}.

\begin{theo}\label{birsurf} Let $(X,\Delta)$ be a klt surface of log general type
 and assume that  the coefficients of $\Delta$ are in a DCC set $\mathcal{A}\subset\q$.
 Then there is a number $N$ depending only on $\mathcal{A}$ such that $\roundup{m(K_X+\Delta)}$ defines a  birational map for $m\ge N$.
\end{theo}  
\begin{proof}
Consider a log resolution $f:X'\la X$ of $(X,\Delta)$ and write 
$f^*(K_X+\Delta)=K_{X'}+(f^{-1})_*\Delta+\sum_ie_iE_i$ with $E_i$ exceptional.
There is a natural number $n$ such that $e_i<1-\frac{1}{n}$ for every $i$. 
Define $\Delta'=(f^{-1})_*\Delta+\sum(1-\frac{1}{n})E_i$. Since we have the inclusion $H^0(X',\roundup{m(K_{X'}+\Delta')})\subset H^0(X,\roundup{m(K_X+\Delta)})$ by replacing the $\mathcal{A}$ with the DCC set $\mathcal{A}\cup\{1-\frac{1}{n}|n\in\mathbb{N}\}$ we can assume that $X$ is smooth.

By  \cite{alex} or  \cite[Theorem 4.8]{almo} we have that vol$(K_X+\Delta)>\alpha^2$ for some $\alpha$ depending 
only on the DCC set $\mathcal{A}$. 
Take  a Zariski decomposition 
$K_X+\Delta \sim_\q A +E $ with $A$ nef  and $E$ effective and $A$ 
orthogonal to each component of $E$.
We have that  vol$(K_X+D)=$vol$(A)>\alpha^2$ 

 Choose two general points $x_1,x_2 \in X$.
 Arguing as in \cite[Lemma 5.4 and Lemma 5.5]{tak} we can produce a divisor
 $D_1 \sim a_1 A$, with $a_1 < \frac {\sqrt{2}}{\alpha}$
 such that there is a non-empty subset
 $I_1$ of $\{1,2\}$ with the following property:
\begin{description}
\item[(*)] $(X,D_1)$ is lc but not klt at $x_i$ for $i \in I_1$ and not lc at $x_i$ for $i \notin I_1$.
\end{description}
With this choice of $a_1$ we can furthermore  assume that
 either codim Nklt$(X,D_1)=2$ at $x_i$ for 
$i\in I_1$ or Nklt$(X,D_1)=Z\cup Z_+$ such that $Z$ is irreducible 
curve and $x_i$ is in $Z$ but not in $Z_+$ for $i \in I_1$.
 
Assuming that $Z\cdot A > c$ for some constant $c$ 
and still following \cite[Lemma 5.8]{tak}
we can produce a divisor $D_2\sim a_2 A$ with $a_2 <c+\epsilon+a_1$ 
such that there is 
a subset $I_2$ of $\{1,2\}$ with the property that $(X,D_2)$
is lc but not klt at $x_i\in I_2$ and not lc at $x_i$ for $i \notin I_2$ and codim Nklt$(X,D_2)=2$ at $x_i$ for $i \in I_2$.

  Now if we set $G=D_2+(m-1-a_2-\epsilon)A+(m-1)E+F$ 
where $0<\epsilon\ll 1$ and $F=\roundup{(m-1)K_X+m\Delta}-(m-1)K_X-(m-1)\Delta$
 we observe that $\roundup{(m-1)K_X+m\Delta}-G\sim_\q\epsilon A$. Since $A$ 
is an ample divisor  Kawamata-Viehweg 
vanishing implies that  $H^1(X,\roundup{m(K_X+\Delta)}\otimes\j(G))=0$ for  
$m>a_2+1$ and hence the linear system $|\roundup{m(K_X+\Delta)}|$ 
gives a birational map onto its image (cf. \cite[Chapter 9]{raz}).

 Thus we can  now assume that for every general point $x\in X$ we have a pair $(D_x,V_x)$, such that $D_x \sim a_1 A$,
 $V_x$ is a pure log canonical centre of 
$D_x$, and dim $V_x = 1$. By 
 \cite[Lemma 3.2]{mac} we  have a diagram 
$$
\xymatrix{
X'\ar[d]^f \ar[r]^\pi & X \\
B&
}
$$
where $\pi$  is dominant and generically finite morphism of normal projective 
 varieties, 
and the image of the general fibre of $f$ is $V_x$ for some $x\in X$.

 Arguing as in section 3 of \cite{tod} we can assume that 
the map $\pi$ is birational. In fact if $\pi$ is not birational we have at
least two centres of log canonical singularities through a general 
point. Replacing each such pair of centres   with a \emph{minimal} centre 
we may assume  that the dimension of the centres is zero and this way   
 $\roundup{m(K_X+\Delta)}$ gives a birational map onto its image 
for $m>3a_1+1$ (compare \cite[page 11]{tod}).

 Thus we consider the case when $\pi$ is birational. We  replace $X$ with a 
model on which $K_X+\Delta$ is nef and big.
To complete the proof we will show that the degree of the restriction of $K_X+\Delta$ to a log canonical centre through a general point on an appropriate model is bounded from below by a constant that depends only on the DCC set $\mathcal{A}$. This is enough since we can apply Kawamata-Viehweg vanishing as before before to produce sections with the desired properties and hence a birational map.

 If $X \la B$ is not a morphism (over a general point $b \in B$ then there is a point $x\in X$ such that we have at least two pairs $(D_1,V_1)$ and $(D_2,V_2)$, such that $D_i \sim a_1 (K_X+\Delta)$,
 $x\in V_i$  a pure log canonical centre of 
$K_X+\Delta+D_i$ of dimension 1 and $V_1\neq V_2$ corresponding to two general fibres of $f$. If $x$ is a smooth point, we have that 
$$(K_X+\Delta)\cdot V_1=\frac{1}{a_1}D_2\cdot V_1 \ge \frac{1}{a_1} V_2\cdot V_1 \ge\frac{1}{a_1}$$
 since $V_1^2\ge 0$.

 If $x$ is not smooth then $(X,\Delta)$ is not terminal at $x$ and so there is a projective birational morphism $\pi:X'\la X$ extracting a divisor of discrepancy less than or equal to zero. Therefore $\pi^*(K_X+\Delta)=K_{X'}+\Delta'$ where $\Delta'\ge 0$ and $K_{X'}+\Delta'$ is still nef and big. Since there are only finitely many divisors of non-positive discrepancy after finitely many extractions as above  we may assume that  there is a morphism $f:X'\la B$. Thus we may  write $\pi^*(K_X+\Delta)=K_{X'}+\Delta'$. Here $\Delta'$ is effective and $K_{X'}+\Delta'$ is  nef and big. 

Now let $\beta=\beta(\mathcal{A})$ be as defined in 3.5 of \cite{almo}. We can assume that every $\pi$-exceptional divisor that dominates $B$ appears with a coefficient 
grater than $1-\beta$ in $K_{X'}+\Delta'$. In fact suppose that this is not the case for an exceptional divisor $E$. Away from the intersection of $E$ with the other components of $\Delta'$ the divisor $E$ intersects two general fibres 
$F_1$ and $F_2$ corresponding to two log canonical centres as before.  With this choice there are no other log canonical places of $K_X+\Delta+D_1+D_2$ 
lying over $\pi(E)$ 
connecting the  intersection of $E$ with $F_1$ and $F_2$. Then by the Connectedness Principle $E$ is a log canonical place for $K_X+\Delta+D_1+D_2$. In particular mult$_E\pi^*D_i\ge \frac{\beta}{2}$  for at least one of the $D_i$, say $D_1$, and hence  $(K_X+\Delta)\cdot V_1>\frac{\beta}{2a_1}$.

 Now take a log resolution $g:X''\la X'$ of $(X',\Delta')$ and let $f'=f\circ g$ 
and write $K_{X''}+\Delta''+\sum e_iE_i+N_1=g^*(K_{X'}+\Delta')+N_2$ 
where $\Delta'''+\sum e_iE_i+N_1$ and $N_2$ are effective with no common components, $\Delta''$ is the strict transform of $\Delta$ and the $E_i$ are the strict transforms of the $\pi$-exceptional divisors that dominate $B$  with $g_*(\Delta''+\sum e_iE_i+N_1)=\Delta'$ ($N_1$ and $N_2$  do not intersect the general fibre of $f'$).   The divisor  
$K_{X''}+\Delta''+\sum E_i+\roundup{N_1}$ is big and $(X'', \Delta''+\sum E_i+\roundup{N_1})$ is lc  with the coefficients in a DCC set that depends only on $\mathcal{A}$ and so 
 by \cite[Theorem 4.6]{almo} the divisor $K_{X''}+\Delta''+(1-\beta)\sum E_i+\roundup{N_1}$ 
is still big.  Hence for the general fibre $F'$ of $f'$ we have that deg$(K_{X''}+\Delta''+(1-\beta)\sum E_i+\roundup{N_1})_{|F'}=\deg (K_{X''}+\Delta''+(1-\beta)\sum E_i)_{|F'}\ge c>0$ where $c$ depends only on $\mathcal{A}$.
Now $g_*(\Delta''+(1-\beta)\sum E_i)\le \Delta'$ and so $(K_{X'}+\Delta')_{|F}\ge c$. Since $K_{X'}+\Delta'=\pi_*(K_X+\Delta)$  it follows that $(K_X+\Delta)\cdot V_1\ge c$.

\end{proof}
 
\begin{cor}  Let $(X,\Delta)$ be a surface klt pair of log general type 
and assume that  the coefficients of $\Delta$ are in a DCC set 
$\mathcal{A}$. Then there is a number $N$ depending only on $\mathcal{A}$ such that $\round{N(K_X+\Delta)}$ defines a  birational map.
\end{cor} 
\begin{proof} Change the coefficients as in the last part of the proof of Theorem \ref{3folds} and reduce to the case in which all the denominators of the coefficients of  $\Delta$ are  the same. Then by taking an appropriate multiple proceed with integral divisors only.
\end{proof}

\end{document}